\documentclass[oneside,11pt]{amsart}
\usepackage[utf8]{inputenc}
\usepackage[english]{babel}
\usepackage{libertine}
\usepackage{floatflt}
\usepackage{comment}
\usepackage{blindtext}
\usepackage{enumitem}
\usepackage{amsthm}
\usepackage{subfig}
\usepackage{listings}
\usepackage{amscd}
\usepackage{centernot}
\usepackage{mathtools}
\usepackage{stmaryrd}
\usepackage{listingsutf8}
\usepackage{setspace}
\usepackage{amsmath}
\usepackage{framed}
\usepackage{minibox}
\usepackage{float}
\usepackage{wrapfig}
\usepackage{longtable}
\usepackage[strict]{changepage}
\usepackage{nicefrac}
\usepackage{bbm}
\usepackage{amssymb}               
\usepackage{amsthm}   
\usepackage{mathtools}
\usepackage{amsfonts}  
\usepackage{mathrsfs}
\usepackage{todonotes}
\usepackage{esvect}
\usepackage{scalerel}
\usepackage{stackengine,wasysym}

\usepackage[toc,page]{appendix}
\usepackage{ragged2e}

\makeatletter
\renewenvironment{thebibliography}[1]
     {\section*{\refname}%
      \normalfont\normalsize 
      \list{\@biblabel{\@arabic\c@enumiv}}%
           {\settowidth\labelwidth{\@biblabel{#1}}%
            \leftmargin\labelwidth
            \advance\leftmargin\labelsep
            \usecounter{enumiv}}%
      \sloppy\clubpenalty4000\widowpenalty4000%
      \sfcode`\.=\@m}
     {\def\@noitemerr{\@latex@warning{Empty `thebibliography' environment}}%
      \endlist}
\makeatother
\usepackage{tikz-cd}
\usetikzlibrary{positioning, angles}
\usetikzlibrary{arrows}
\usetikzlibrary{decorations.pathmorphing, decorations.pathreplacing}
\usetikzlibrary{decorations.markings}
\usepackage{mathtools}
\usepackage{lipsum}
\usepackage{xfp}

\usepackage{tikz}
\usepackage{tikz-3dplot}
\usetikzlibrary{3d, angles, arrows.meta, calc, decorations.markings, 
                decorations.pathmorphing, decorations.pathreplacing, 
                positioning, shapes,shadings}
\usepackage{pgfplots}
\usepackage{pgfmath}
\usepackage{xfp}
\usepackage{hyperref}
\tikzset{
  on each segment/.style={
    decorate,
    decoration={
      show path construction,
      moveto code={},
      lineto code={
        \path [#1]
        (\tikzinputsegmentfirst) -- (\tikzinputsegmentlast);
      },
      curveto code={
        \path [#1] (\tikzinputsegmentfirst)
        .. controls
        (\tikzinputsegmentsupporta) and (\tikzinputsegmentsupportb)
        ..
        (\tikzinputsegmentlast);
      },
      closepath code={
        \path [#1]
        (\tikzinputsegmentfirst) -- (\tikzinputsegmentlast);
      },
    },
  },
  mid arrow/.style={postaction={decorate,decoration={
        markings,
        mark=at position .5 with {\arrow[#1]{stealth}}
      }}},
}
\usepackage[toc,page]{appendix}
\usepackage{ragged2e}
\usepackage{fancyhdr}

\usepackage{fancyhdr}\oddsidemargin = 2pt       \topmargin = 0pt\headheight = 12pt          \headsep = 25pt\textheight = 609pt         \textwidth = 424pt\marginparsep = 11pt        \marginparwidth = 54pt\footskip = 30pt            \marginparpush = 5pt \hoffset = 0pt              \voffset = 0pt\paperwidth = 597pt         \paperheight = 800pt
\def\Xint#1{\mathchoice 
  {\XXint\displaystyle\textstyle{#1}}%
  {\XXint\textstyle\scriptstyle{#1}}%
  {\XXint\scriptstyle\scriptscriptstyle{#1}}%
  {\XXint\scriptscriptstyle\scriptscriptstyle{#1}}%
  \!\int}
\def\XXint#1#2#3{{\setbox0=\hbox{$#1{#2#3}{\int}$} 
  \vcenter{\hbox{$#2#3$}}\kern-.5\wd0}}
\def\dashint{\Xint-}

\newcommand{\R}{\ensuremath{\mathbb{R}}}

\newcommand{\norm}[1]{{\ensuremath{||{#1}||}}}

\let\theta\vartheta
\let\phi\varphi
\let\epsilon\varepsilon
\let\bar\overline

\theoremstyle{plain}
\newtheorem{thm}{Theorem}[section]
\newtheorem*{thm*}{Theorem}
\newtheorem{prop}[thm]{Proposition}
\newtheorem{constr}[thm]{Construction}

\newtheorem{lem}[thm]{Lemma}

\theoremstyle{definition}
\newtheorem*{defn*}{Definition}
\newtheorem{defn}[thm]{Definition}

\newtheorem{oss}[thm]{Remark}
\newtheorem{example}{Example}
\newread\tmp

\usepackage{etoolbox}
\makeatletter
\providecommand{\subtitle}[1]{
  \apptocmd{\@title}{\vspace*{0.2cm}\par {\small\normalfont #1 \par}}{}{}
}
\makeatother
\counterwithout*{equation}{section}
\def\pri{\hbox to 10pt{\hfil\hbox to 0.4pt{\vrule height5pt width0.4pt
                 depth0pt}\vrule width5pt height0.4pt depth0pt\hfil}}
\usepackage{units}
\newcommand{\BV}{\mathop{\rm BV}\nolimits}

\addto\captionsenglish{}

\makeatletter
\newcommand{\tpitchfork}{%
  \vbox{
    \baselineskip\z@skip
    \lineskip-.52ex
    \lineskiplimit\maxdimen
    \m@th
    \ialign{##\crcr\hidewidth\smash{$-$}\hidewidth\crcr$\pitchfork$\crcr}
  }%
}
\makeatother
\newcommand{\gn}{{\bf n}}
\newcommand{\gu}{{\bf u}}
\usepackage{xcolor}
\usepackage{graphicx}
\usepackage{listings}
\usepackage{inconsolata}  
\numberwithin{equation}{section}
\newcommand{\gt}{{\bf t}}
\newcommand{\gnu}{{\bf \nu}}

\date{}
\usepackage{tikz}
\usetikzlibrary{3d, calc, shadings}

\title[spherical elastica]{ Weak elastic energy of rectifiable curves in the sphere}

\author[D. Mucci]{Domenico Mucci} 
\address{Domenico Mucci\\ Dipartimento di Scienze Matematiche, Fisiche e Informatiche, Universit\`a di Parma\\ Campus - Parco Area delle Scienze 53/A, 43124 Parma, Italy}
\email{\url{domenico.mucci@unipr.it}}

\author[A. Saracco]{Alberto Saracco} 
\address{Alberto Saracco\\ Dipartimento di Scienze Matematiche, Fisiche e Informatiche, Universit\`a di Parma\\ Campus - Parco Area delle Scienze 53/A, 43124 Parma, Italy}
\email{\url{alberto.saracco@unipr.it}}

\author[C. Sopio]{Cristian Sopio}  
\address{Cristian Sopio\\Dipartimento di Scienze Matematiche, Fisiche e Informatiche, Universit\`a di Parma\\ Campus - Parco Area delle Scienze 53/A, 43124 Parma, Italy}  \email{\url{cristian.sopio@unipr.it}}

\makeatother
\begin{document}
\nocite{MS2}
\onehalfspacing
\maketitle

\begin{abstract}
    We introduce for any exponent $p>1$ the $p$-curvature functional for rectifiable curves in the two-dimensional sphere. We prove that this functional is finite and agrees with the integral of the geodesic curvature raised to the power $p$ on curves whose arc length parameterization is in the Sobolev class $W^{2,p}$.
\end{abstract}

\section{Introduction}
This work is a first approach to the project of studying
the $p$-\emph{curvature} of \emph{irregular curves} in Riemannian surfaces $M$. For a comprehensive introduction to irregular curves, we refer to \cite{alexandrov2012general,Re,Mil,Mi53}.

In \cite{Mil}, it has been introduced the definition of \emph{total curvature} $\mathrm{TC}(c)$ for a rectifiable curve $c$
 as the supremum of the \emph{rotation} of inscribed polygonals.

An {\em intrinsic} theory of rectifiable curves with finite total curvature in a Riemannian surfaces $M$ has been recently developed by the first two authors in \cite{MS2}  for the plastic case ($p=1$). Here we want to address, following \cite{MS1}, the elastic case 
 ($p>1$) for $M=\mathbb{S}^2$, the unit sphere in $\mathbb R^3$.
When the sectional curvature is positive, the expected monotonicity formula fails to hold, see Remark~\ref{monot}. To obtain a good intrinsic notion of total curvature $\mathrm{TC}_M(c)$, we need to use the \emph{modulus} $\mu_c(P)$ of an inscribed polygonal $P$ introduced in \cite{alexandrov2012general}, that is equal to the maximum of the geodesic diameter of the arcs of $c$ determined by the consecutive vertices in $P$. By \cite{C}, notice that a representation formula holds for the total curvature of piecewise smooth curves $c$  
\begin{equation}\label{ismooth} 
    \mathrm{TC}_M(c)=\int_c|k_M(c)|\,ds +\sum_i|\theta_i|, 
\end{equation}
where $|\theta_i|\in [0,\pi]$ are the \emph{turning angles} at the corner points of the curve $c$ and $k_M(c)$ is the curvature vector at smooth points of $c$.

The plastic case is completely different from the elastic one, even in the Euclidean setting. Indeed, as we can see from \eqref{ismooth}, a curve with finite total curvature may have 
corner points. In the elastic case, as it is pointed out in \cite[Section 4b]{MS1}, if the \emph{total} $p$-\emph{curvature} is finite, the curve turns out to be in the Sobolev space $W^{2,p}(I,\R^n)$. Also in the spherical setting, a rectifiable curve in $\mathbb{S}^2$ of finite total curvature may have corner points, but as observed in Remark~\ref{NoCorners}, the curve in the elastic case cannot have any corner points.

We recall that for general rectifiable curves with finite total curvature in a Riemannian surface $M$, a representation formula for the total curvature holds true using the language of functions of bounded variation, see \cite{MS2}.  
 Here, for $p>1$ we obtain the expected Sobolev regularity.
\subsection{Results}
We concentrate on studying the integral of the geodesic curvature raised to the power $p$ for rectifiable curves in the sphere $\mathbb{S}^2$. We extend the results in \cite{MS1} obtained by the first two authors in the Euclidean setting that we expect to hold in 
 generic compact Riemannian surfaces.

Our main result is to obtain a geometric functional $\mathcal{F}_p(c)$ on a rectifiable curve $c$ parametrized by arc length that relies on the Sobolev regularity of the curve. It turns out that if $\mathcal{F}_p(c)<\infty$ the curve is in $W^{2,p}(I,\mathbb{S}^2)$ and $\mathcal{F}_p(c)$ coincides   with the integral of the $p$-power of  the modulus of the geodesic curvature.

In the same spirit as Lebesgue-Serrin's relaxed functional, we introduce the {\em $p$-curvature} functional $\mathcal{F}_p(c)$ of rectifiable curves $c$ in $\mathbb{S}^2$ as
$$ \mathcal{F}_p(c):=\inf\Bigl\{\liminf_{h\to\infty}\mathbf{k}_p(P_h)\mid \{P_h\}\ll c\,,\,\,\mu_c(P_h)\to 0\Bigr\}\quad\mbox{ for }p> 1,$$
where the $p$-\emph{rotation} $\mathbf{k}_p(P)$ of an inscribed polygonal $P$ is obtained distributing the turning angles at corner points of $P$ via an optimal curve $\gamma(P)$ of locally constant geodesic curvature and integrating its geodesic curvature raised to power $p$, see Section~\ref{Sec:pcurv}, Definition~\ref{F_p} and Appendix~\ref{AA} for the construction of the curve $\gamma(P)$.
\begin{thm}
     Let $c:[0,L]\rightarrow\mathbb{S}^2$ be a rectifiable and open curve in $\mathbb{S}^2$ parametrized in arc-length,  and let $p>1$. Then $$\mathcal{F}_p(c)<\infty 
\Longleftrightarrow c\in W^{2,p}([0,L],\mathbb{S}^2)$$ and in this case, there holds $$\mathcal{F}_p(c)=\int_c|k_{\mathbb{S}^2}|^p\,ds=\int_0^L\norm{\Ddot{c}^\top}^p\,ds.$$
\end{thm}
 Here $\ddot c^\top$ stands for the projection of the Euclidean second derivative vector of the curve $\quad$ $c:I\rightarrow\mathbb{S}^2\subseteq\R^3$ onto the tangent plane $T_{c(s)}\mathbb{S}^2$, see Section~\ref{Sec:geod.curv}.

\subsection{Comments and further directions}
Comparing this work with \cite{MS1}, one can see that for a good approximation sequence $\{P_h\}$ of a curve $c:I\rightarrow\mathbb{S}^2$ the $p$-rotation behaves 
 like in the Euclidean case, more precisely $\mathbf{k}_p(P_h)\sim\sum_i\ell^{1-p}\theta_i^p$. We expect that the same holds for curves in the Hyperbolic plane and, more general, for curves in $\mathrm{CAT}(k)$ two dimensional Riemannian manifolds.

In the proofs, we take advantage of local comparisons with planar curves obtained by means of conformal images of pieces of the spherical curve.
\subsection*{Acknoledgements}
The authors were partially supported by the GNAMPA of INDAM.

\section{Preliminaries}
\subsection{Length and inscribed polygonals}
Consider a curve $c$ in the sphere parameterized by the continuous map
$c: I\to\mathbb{S}^2$, where $ I=[0,L]$.
We say that $P$ is a \emph{polygonal} curve
{\em inscribed} in $c$, say $P \ll c$, if it is obtained by choosing a
finite partition $\mathcal{P}\coloneqq\{0=t_0<\ldots<t_{h}=L\}$ of $I$, say $P=P(\mathcal{P})$,  such that
$P(t_i)=c(t_i)$ for $i=0,\ldots,n$, and $P_{|_{[t_i,t_{i+1}]}}$ is 
 the geodesic arc of $\mathbb{S}^2$ that joins $c(t_i)$ to $c(t_{i+1})$, if the latter are not antipodal points. The length of the polygonal is
$$\mathcal{L}(P)=\sum_{i=0}^{h-1}d_{\mathbb{S}^2}(c(t_{i+1}),c(t_{i}))\,.$$

We denote
$$\mathrm{mesh}\, \mathcal{P}:=\sup_{0\leq i\leq h-1}|t_{i+1}-t_{i}|\,,\quad \mathrm{mesh}\, P:=\sup_{0\leq i\leq h-1}
d_{\mathbb{S}^2}(c(t_{i+1}),c(t_{i})) \,. $$

The {\em length} $\mathcal{L}(c)$ of the curve $c$ is defined by
$$\mathcal{L}(c):=\sup\{\mathcal{L}(P)\mid P\ll c\} $$
and $c$ is said to be {\em rectifiable} if $\mathcal{L}(c)<\infty$.
 In that case, 
taking $P_h=P(\mathcal{P}_h)$, where
$\{\mathcal{P}_h\}$ is any sequence of partitions of $I$ such that
$\mathrm{mesh}\, \mathcal{P}_h\to 0$,
by uniform continuity we get $\mathrm{mesh}\, P_h\to 0$ and the
convergence $\mathcal{L}(P_h)\to\mathcal{L}(c)$ of the length functional.
Finally, if a curve is rectifiable, its arc length parametrization $c(s)$ is Lipschitz continuous and hence, by Rademacher's theorem, 

the derivative $\dot{c}$ exists almost everywhere.

\begin{defn}
    The {\em Fr\'echet distance} $d(c_1,c_2)$ between two rectifiable curves is the infimum, over all strictly monotonic
reparameterizations, of the maximum pointwise distance.
\end{defn} 
 
Therefore, if $d(c_1,c_2)=0$, the two curves are equivalent in the following sense: homeomorphic reparameterizations that approach the infimal value zero will limit to the more general reparameterization that might eliminate or introduce intervals of constancy, {compare \cite{Su}.}
 
\par Moreover, if $\{c_h\}$ is a sequence of rectifiable curves in $\mathbb{S}^2$ such that $d(c_h,c)\to 0$ as $h\to \infty$ for some rectifiable curve $c$, then by lower semicontinuity
 
\begin{equation}
    \mathcal{L}(c)\leq\liminf_{h\to\infty}\mathcal{L}(c_h)\,.
\end{equation} 
\subsection{Total intrinsic curvature}
In this section we point out the difference in the plastic case between 
 the cases of curves in the Euclidean space and in an immersed surface in $\R^3$. Moreover, we observe that for surfaces with positive sectional curvature, the expected monotonicity formula on the rotation of polygonals does not hold.
\subsubsection{Euclidean total curvature}In the Euclidean setting, given a rectifiable curve $c$ and an inscribed polygonal $P$, we recall that its rotation $\mathbf{k}_{\R^n}^*(P)$  is the sum of the exterior angles between consecutive segments. 
\begin{oss}\label{monot}
The following facts hold:
    \begin{itemize}
    \item if $P$ and $P'$ are inscribed polygonals and $P'$ is obtained by adding a vertex in $c$ to the vertexes of $P$, then $\mathbf{k}_{\mathbb{R}^n}^*(P)\leq \mathbf{k}_{\mathbb{R}^n}^*(P')$\,;
 
\item if $c$ has finite total curvature, for each point $p$ in $c$, small open arcs of $c$ with an end point equal to $p$ have small total curvature. \end{itemize} 
\end{oss}
By the above remark, starting with any polygonal $P$ inscribed in $c$ and performing a sequence $P_h$ by adding vertices to $P$, we get that the rotation of $P_h$ increases and the polygonal sequence converges in length to the curve,
 if the mesh of the polygonals goes to zero. Then, the total curvature can be defined as the supremum of the rotation of any polygonal sequence $P_h$ with $\mathrm{mesh}(P_h)\rightarrow 0$.

\subsubsection{Geodesic curvature}\label{Sec:geod.curv}
Assume now that $c$ is a smooth and regular curve supported in $\quad$ $M\subseteq\R^3$ parametrized by arc length. The Darboux frame along $c$ is the triad $(\gt,\gn,\gu)$,
where $\gt(s):=\dot c(s)$ is the unit tangent vector, $\gn(s):=\gnu(c(s))$, $\gnu(p)$ being the  oriented unit normal to the tangent 2-space $T_pM$, and $\gu(s):=\gn(s)\times \gt(s)$, where $\times$ denotes the vector product in $\R^3$, is the unit conormal. Therefore, the tangent space $T_{c(s)}M$ is spanned by $(\gt(s),\gu(s))$.

The curvature vector $\Ddot{c}=\dot\gt(s)$ is orthogonal to
$\gt(s)$, and thus decomposes as
$$\Ddot{c}(s)=k_M(s)\,\gu(s)+k_n(s)\,\gn(s) $$
where $k_M:=\Ddot{c}\cdot\gu$ and $k_n:=\Ddot{c}\cdot\gn$ denote the {\em geodesic} and the {\em normal curvature} of $c$ respectively and $\cdot$ the scalar product in $\R^3$. The projection $k_M\gu$ of $\Ddot{c}$ onto the tangent bundle of $M$ is an intrinsic object, see \cite{MS2}, and for the sake of readability will be denoted by $\Ddot{c}^\top$. If $c$ is a geodesic on $M$, we have $k_M\equiv 0$,
 whereas in general
\begin{equation}
    |k_M|=|\Ddot{c}\cdot\gu|=\|\Ddot{c}^\top\|,\quad k_M=k_M(c)\,.
\end{equation}
\subsubsection{Total (geodesic) intrinsic curvature} 
In this section, we recall some properties concerning the total intrinsic curvature of smooth curves contained into surfaces. We thus let $M$ denote an immersed surface in $\R^3$. We assume $M$ smooth (at least of class $\mathcal{C}^3$), closed, and compact, our model case being $M=\mathbb{S}^2$, the standard unit sphere in $\R^3$. The {\em (intrinsic) rotation}
$\mathbf{k}_M^*(P)$ of a polygonal $P$ in $M$, where $M\subseteq\R^3$, is the sum of the turning angles between the consecutive geodesic arcs of $P$.

The following property has been proved in \cite{C}.

\begin{thm}\label{Tdens}{\bf (\cite[Thm.~3.4]{C})}
     Let $c$ be a regular curve in $M$ of class $C^2$, parameterized by arc-length. Then, for any sequence $\{P_h\}$ inscribed in $c$ such that $\mathrm{mesh} (P_h)\to 0$, one has
$$ \lim_{h\to\infty}\mathbf{k}_M^*(P_h)=\int_0^L|k_M(c(s))|\,ds\,. $$
\end{thm}

\par As a consequence, for a curve $c$ in $M$, one is tempted to define its total intrinsic curvature as in the Euclidean case, i.e., as the supremum of the intrinsic rotation
$\mathbf{k}_M^*(P)$ computed among all the polygonals $P$ inscribed in $c$. However, as observed in \cite{C}, if $M$ has positive sectional curvature, as e.g. $M=\mathbb{S}^2$, the latter definition does not work. In fact, if $P,P'\ll c$, and $P'$ is obtained by adding a vertex in $c$ to the vertexes of $P$, then the monotonicity inequality $\mathbf{k}_{M}^*(P)\leq \mathbf{k}_{M}^*(P')$ holds true in general provided that $M$ has non-positive sectional curvature.
In fact, it relies on the fact that in this case the sum of the interior angles of a geodesic triangle of $M$ is not greater than $\pi$, see \cite[Lemma~4.1]{C}.

\begin{example}
    In the sphere $M=\mathbb{S}^2$, if we take as $c$ a parallel that is not a great circle and $P,P'$ are inscribed polygonals such that $P'$ is obtained by adding a vertex to $P$, the other inequality holds $$\mathbf{k}^*_{\mathbb{S}^2}(P)>\mathbf{k}^*_{\mathbb{S}^2}(P') \quad \mbox{ and }\quad\mathbf{k}^*_{\mathbb{S}^2}(P)>\int_c |k_{\mathbb{S}^2}(c)|\, ds.$$
\end{example}

\par Actually, the good definition turns out to be the one introduced by Alexandrov-Reshetnyak in \cite{alexandrov2012general} using the {\em modulus} $\mu_c(P)$ of a polygonal $P$ inscribed in $c$, that is the maximum of the geodesic diameter of the arcs of $c$ determined by two consecutive vertexes in $P$.

For $\epsilon>0$, we let
$$ \Sigma_\epsilon(c):=\{ P\ll c\mid \mu_c(P)<\epsilon\}\,. $$
 
\begin{defn}\label{Dcurv}
     The {\em total intrinsic curvature} of a curve $c$ in $M$ is
$$ \mathrm{TC}_M(c):=\lim_{\epsilon\to 0^+}\sup\{ \mathbf{k}_M^*(P)\mid P\in \Sigma_\epsilon(c)\}\,. $$ 
\end{defn}
 
\par Clearly, the above limit is equal to the infimum of $\sup\{ \mathbf{k}_M^*(P)\mid P\in \Sigma_\epsilon(c)\}$ as $\epsilon>0$. Moreover, arguing as in \cite[Prop.~2.1]{ML}, for a polygonal $P$ in $M$ we always have $$\mathrm{TC}_M(P)=\mathbf{k}_M^*(P).$$

Most importantly, making use of a result by Dekster \cite{D}, as a consequence of \cite[Prop.~2.4]{ML} one obtains:
 
\begin{prop}\label{Pappr} The total curvature $\mathrm{TC}_M(c)$ of any curve $c$ in $M$ is equal to the limit of the rotation $\mathbf{k}_M^*(P_h)$ of {\em any} sequence of polygonals $\{P_h\}$ inscribed in $c$ such that
$\mu_c(P_h)\to 0$.     
\end{prop}

\par Proposition~\ref{Pappr} fills in the gap given by the lack of monotonicity, yielding to the conclusion that
Definition~\ref{Dcurv} involves a control on the modulus and not on the mesh, at least when the sectional curvature of $M$ fails to be non-negative.
 
\par As a consequence, by Theorem~\ref{Tdens} one infers that for smooth curves $c$ in $M$ one has $$\mathrm{TC}_M(c)=\int_c|k_M(c)|\,ds.$$
 
By \cite[Cor.~3.6]{C}, for piecewise smooth curves $c$ in $M$ one similarly obtains that
\begin{equation}
    \label{TCMsmooth} \mathrm{TC}_M(c)=\int_0^L|k_M(c)|\,ds+\sum_i|\theta_i|\,. 
\end{equation}
 
In this formula, the integral is computed separately outside the corner points of $c$, where the geodesic curvature $k_M$ is well-defined,
and the second addendum denotes the finite sum of the absolute value of the oriented turning angles $\theta_i$ between the incoming and outcoming unit tangent vectors at each corner point of $c$.
Therefore, for piecewise smooth curves in $M\subseteq\R^3$ we can rewrite formula \eqref{TCMsmooth} as
 
\begin{equation}
    \label{TCMsmoothV} \mathrm{TC}_M(c)= \int_0^L\|\Ddot{c}^\top(s)\|\,ds +
\sum_{s\in J_\gt}d_{\mathbb{S}^2}(\gt(s+),\gt(s-))\,. 
\end{equation}
where $J_\gt$ is the jump set of the $\BV$ function $\gt$ and $\ddot c^\top$ is the tangential part of the second derivative of $c$.

Equality \eqref{TCMsmooth} shows that a rectifiable curve with finite total curvature may have corners in the plastic case. In the elastic case, this is not true.
\subsection{Functional setting}
We deal with Sobolev maps in the two-sphere $\mathbb{S}^2$.
For any $p>1$ and for $k=1,2$, we define the Sobolev class 
$$
W^{k,p}(I,\mathbb{S}^2)\coloneqq\left\{ u\in W^{k,p}(I,\R^3)\,|\, \|u\|=1 \mbox{ a.e. on }I
\right\}\,.
$$ 
We will work with curves $u$ parametrized by arc length. If $u:I\rightarrow\mathbb{S}^2\subseteq\R^3$ is $\mathcal{C}^2$ and $\|\dot u\|=1$, the curvature $k_{\mathbb{S}^2}(u)$ coincides with the projection on the tangent space in $u(t)$ of the second derivative $\ddot u$, namely 
\begin{equation}\label{Spher.curv}
    |k_{\mathbb{S}^2} (u)|=\|\ddot u^\top\|,
\end{equation}
where 
$$\ddot u^\top=\ddot u+u\ .$$ 
Therefore, asking for the covariant derivative $\nabla_{\dot u}\dot u$ to be in $L^p$ is equivalent to the curve being in $W^{2,p}(I,\mathbb{S}^2)$.

We observe that, by standard Sobolev embeddings $W^{2,p}(I,\mathbb{S}^2)\hookrightarrow \mathcal{C}^1(I,\mathbb{R}^3)$, a curve $u\in W^{2,p}(I,\mathbb{S}^2)$ is at least $\mathcal{C}^1$ smooth.

\section{$p$-curvature functional}\label{Sec:pcurv}

\subsection{$p$-rotation of polygonals}
In this section, we deal with the definition of $p$-rotation for polygonals and the $p$-curvature functional for rectifiable curves in the sphere $\mathbb{S}^2$. We first recall how they can be defined in the Euclidean case and then, we perform an analogous construction in the spherical setting.
\subsubsection{The Euclidean case}
From a given polygonal $P$ inscribed in $c:I\rightarrow \R^n$, the first two authors in \cite{MS1} construct a curve $\gamma(P)$ that is piecewise smooth and insisting on $P$. The idea is to redistribute the \emph{curvature measure} concentrated in $\{t_1,\ldots,t_{h-1}\}$ along pieces of smooth curves with constant curvature, i.e., pieces of circles.
\begin{example}
    Let $P$ be a polygonal with a vertex in $(0,0)\in \R^2$ and for $t\in[0,\epsilon]$, the two segments $(t,0)$ and $(\cos(\alpha)\,t,\sin(\alpha)\,t)$, where $\alpha\in(0,\pi)$ and $t\in[0,\epsilon]$.
    For $\epsilon>0$, denoting $s=t/\epsilon\tan(\alpha/2)$, the curve $$\gamma(s)=(\epsilon,\epsilon\tan(\alpha/2))+\epsilon\tan(\alpha/2)\left(\cos (s),\sin (s)\right)\,,$$ 
 where $s\in[\pi/2+\alpha,\,3\pi/2]$, coincides at order $1$ with the end points of the two segments. If $\theta\coloneqq\pi-\alpha$ is the turning angle, then 
 $$\int_\gamma |k_{\R^2}(\gamma)|\,ds=\theta, \quad \int_\gamma |k_{\R^2}(\gamma)|^p\,ds=\epsilon^{1-p}\theta\tan^{p-1}(\theta/2).$$
 
 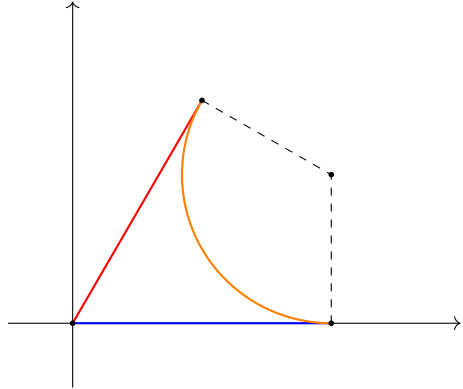
\begin{figure}[H]
    \centering
   \begin{tikzpicture}[scale=1.7]
  \def\eps{2}
  \def\alpha{60} 
  \def\alpharad{\fpeval{\alpha*pi/180}} 
  \def\tanAlpha{\fpeval{tan(\alpharad)}}
  \def\tanAlphaHalf{\fpeval{tan(\alpharad/2)}}
  \def\cosAlpha{\fpeval{cos(\alpharad)}}
  \def\gammaTwoX{\fpeval{\eps*\cosAlpha}}
  \def\gammaTwoY{\fpeval{\gammaTwoX*\tanAlpha}}
  \def\centerY{\fpeval{\eps*\tanAlphaHalf}}
  \def\radius{\fpeval{\eps*\tanAlphaHalf}}
  \def\startAngle{\fpeval{90+\alpha}} 
  \def\endAngle{270} 

  \draw[->] (-0.5,0) -- (3,0);
  \draw[->] (0,-0.5) -- (0,2.5);

  \draw[thick,blue] (0,0) -- (\eps,0);

  \draw[thick,red] (0,0) -- ({\gammaTwoX},{\gammaTwoY});

  \draw[dashed] ({\gammaTwoX},{\gammaTwoY}) -- (\eps,{\centerY});

  \draw[dashed] (\eps,0) -- (\eps,{\centerY});

  \draw[thick,orange,domain=\startAngle:\endAngle,samples=100,variable=\u]
    plot ({\eps + \radius*cos(\u)}, {\centerY + \radius*sin(\u)});

  \filldraw (0,0) circle (0.5pt);
  \filldraw (\eps,0) circle (0.5pt);
  \filldraw ({\gammaTwoX},{\gammaTwoY}) circle (0.5pt);
  \filldraw (\eps,{\centerY}) circle (0.5pt);

\end{tikzpicture}
    \caption{The curvature $k_{\R^2}(\gamma)$ is the inverse of the radius $R=\tfrac{\epsilon}{\tan{(\theta/2)}}$ of the orange curve $\gamma$.}
    
\end{figure}
\end{example}
Then, they define the $p$-rotation of  an equilateral polygonal $P$ inscribed in a curve $c:I\rightarrow\R^n$ with  edge length $\ell$, vertexes $P(t_i)$ and turning angles $\theta_i$ for $i=1,\ldots,h-1$ as $$\mathbf{k}_p(P)\coloneqq\sum_{i=1}^{h-1}(\ell/2)^{1-p}\theta_i\tan^{p-1}(\theta_i/2).$$
It turns out that $\mathbf{k}_1(P)=\mathbf{k}_{\R^n}^*(P)=\mathrm{TC}_{\R^n}(P),$ so this definition of $p$-rotation includes the definition of \emph{total curvature} for $p=1$.

\subsubsection{The spherical case}
From now on, we consider the case of $M=\mathbb{S}^2$.

We define a $p$-curvature functional on rectifiable spherical curves that gives a notion of \emph{total} $p$-\emph{curvature}. 
 
Given a rectifiable curve $c:I:\rightarrow\mathbb{S}^2$ parametrized by arc length and an inscribed polygonal $P\ll c$, using Appendix~\ref{AA}, we define the curve $\gamma(P)$ as follows.
\begin{itemize}
\item Around every vertex $P(t_i)$, we look at the length of the edges of $P$ that insist to the corner. We start by looking at $P(t_1)$, let $\bar \ell_0,\bar \ell_1$ be the length of $P_{|_{[t_0,t_1]}},P_{|_{[t_1,t_2]}}$ respectively. Let $\ell_1=\min\{\bar \ell_0,\bar \ell_1\}$ and perform the construction given in Appendix~\ref{AA} on $\ell_1/2$;
    \item repeat the construction on the next vertex $P(t_2)$, and so on, using as $\ell_i=\min\{\bar \ell_{i-1},\bar\ell_i\}$, where $\bar\ell_i$ is the length of the polygonal restricted to $[t_{i},t_{i+1}]$. ;
    \item take as $\gamma(P):[0,L]\rightarrow\mathbb{S}^2$ the obtained curve where near to the corners we glue the previous construction with the polygonal.
    
    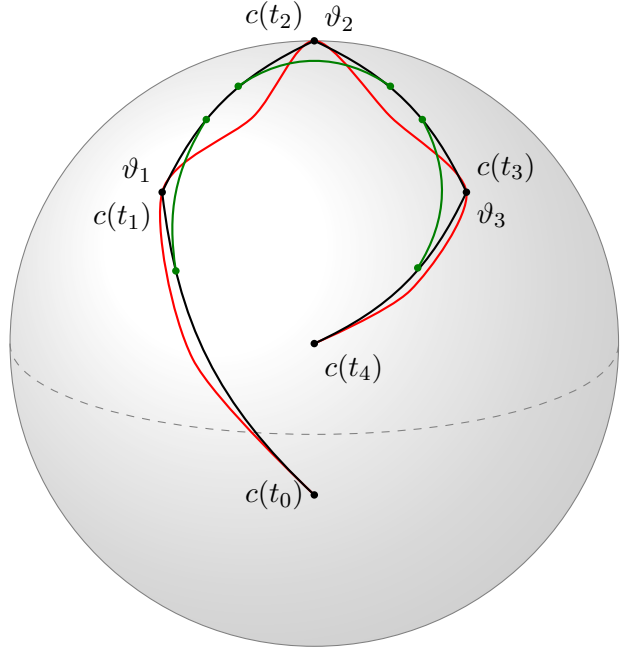
\begin{figure}[h]
\centering
\begin{tikzpicture}[scale=2]
    \shade[ball color=white!10, opacity=0.3] (0,0) circle (2cm);

    \draw[gray] (0,0) circle (2cm);
    \draw[dashed,gray] (-2,0) arc (180:360:2 and 0.6);
 
    \coordinate (N) at (0,2);                 
    \coordinate (A) at (1,1);         
    \coordinate (B) at (-1,1);        
    \coordinate (C) at (0,0);             
    \coordinate (D) at (0,-1);            
    \coordinate (E) at (-0.8,-0.1);
    \coordinate (F) at (-0.4,1.5); 
      \coordinate (G) at (.5,1.5);   
      \coordinate (H) at (0.63,.35);  
\draw[thick, red, smooth] plot coordinates {(D) (E) (B) (F) (N) (G)(A) (H) (C)};

    \coordinate (MNA) at (.5,1.7);
    \coordinate (MAB) at (-0.5,1.7);
     \coordinate (MNBD) at (-0.71,1.48);
    \coordinate (MBD) at (-0.91,0.48);
    \coordinate (MNAC) at (0.71,1.48);
     \coordinate (MCE) at (0.68,0.5);

    \draw[black, thick] (N) to[bend left=20] (A);  
    \draw[black, thick] (A) to[bend left=20] (C);
    \draw[black, thick] (N) to[bend right=20] (B);
    \draw[black, thick] (B) to[bend right=20] (D);

    
    \draw[green!50!black, thick] (MNA) to[bend right=35] (MAB);
    \draw[green!50!black, thick] (MNBD) to[bend right=20] (MBD);
    \draw[green!50!black, thick] (MNAC) to[bend left=30] (MCE);

    \foreach \p in {MNA,MAB,MBD,MNBD,MNAC, MCE} {
        \fill[green!50!black] (\p) circle (0.7pt);    }

    \foreach \p in {N,A,B,C,D} {
        \fill[black] (\p) circle (0.7pt);
    }

    \node[above left] at (N) {$c(t_2)$};
    \node[above right] at (N) {$\theta_2$};
    \node[below left] at (B) {$c(t_1)$};
    \node[above left] at (B) {$\theta_1$};
    \node[left] at (D) {$c(t_0)$};
    \node[above right] at (A) {$c(t_3)$};
    \node[below right] at (A) {$\theta_3$};
    \node[below right] at (C) {$c(t_4)$};

\end{tikzpicture}
\caption{Construction applied to a geodesic polygonal with $h=4$. In red the curve $c(t)$, in black the polygonal with vertex $\{c(t_1),c(t_2),c(t_3)\}$, respectively turning angles $\theta_1,\theta_2,\theta_3$ and in green the curve $\gamma(P)$.}
\end{figure}

\end{itemize}

\begin{defn}
    Given a spherical polygonal $P$, we define the $p$-rotation of $P$ as $$\mathbf{k}_p(P)\coloneqq\int_{\gamma(P)}|k_{\mathbb{S}^2}(\gamma(P))|^p\,dt.$$
\end{defn}
The curve $\gamma(P)$ that we obtained is smooth except for $2h-2$ points where it is  $\mathcal{C}^1$. Moreover by \cite{AceMuc16}, it turns out that $$\mathbf{k}_1(P)\coloneqq\int_{\gamma(P)}|k_{\mathbb{S}^2}(\gamma(P))|\,dt=\int_{0}^{\mathcal{L}(\gamma(P))}\norm{\Ddot{c_P}^\top(s)}\,ds,$$ where $c_P$ is the arc length parametrization of the curve $\gamma(P).$ 

Here, it is no longer true, as in the Euclidean case, that $\mathbf{k}_1(P)=\mathrm{TC}_{\mathbb{S}^2}(P)$, but for sequences $\{P_h\}\ll c$ with $\mu_c(P_h)\rightarrow 0$, we have $\mathbf{k}_1(P_h)\sim\mathrm{TC}_{\mathbb{S}^2}(c)$.

We define the $p$-curvature of the curve $c$ as the Lebesgue–Serrin relaxed functional, of $\quad$ $p$-rotations of polygonals inscribed in the curve.  
\begin{defn}\label{F_p}
    The $p$-\emph{curvature} functional of a rectifiable curve $c$ is defined as $$ \mathcal{F}_p(c)\coloneqq\inf_\epsilon \left\{\liminf_{h\rightarrow \infty}{\mathbf{k}_p(P_h)}:\{P_h\}\subseteq\Sigma_\epsilon(c)\right\},$$
\end{defn}

It turns out that for an equilateral polygonal $P$ with $h$ edges of length $\ell$ and turning angles $\theta_1,\ldots,\theta_{h-1}$, the $p$-rotation $\mathbf{k}_p(P)$ is given by \begin{equation}
\mathbf{k}_p(P)\coloneqq\sum_{i=1}^{h-1}2\arctan\left(\frac{\Psi(\theta,\ell)}{\cos(\theta_i/2)\cos(\ell/2)}\right)\cdot \frac 1{\Psi(\theta,\ell)}\cdot \frac{\sin^p(\theta_i/2)}{(\cos(\theta_i/2)\sin(\ell/2))^{p-1}}\,,    
\end{equation} where $\Psi(\theta,\ell)\coloneqq\sqrt{\sin^2(\ell/2)+\sin^2(\theta/2)\cos^2(\ell/2)}$. For an explicit computation, see  Appendix~\ref{AA}.

\begin{oss}

We observe that the spherical $p$-rotation of a polygonal $P$, for $\theta_i$ and $\ell$ small, is almost identical to the Euclidean $p$-rotation, indeed $$\mathbf{k}_p(P)\sim2\sum_{i=1}^{h-1}\tan^p(\theta_i/2)\sin^{1-p}(\ell/2)\sim\sum_{i=1}^{h-1}\theta_i^p\ell^{1-p}.$$
This will be quantitatively used to prove Theorem~\ref{MainThm} through the local conformality of the sphere to the plane explored in Proposition~\ref{PropPLA}. 
\end{oss}

\section{Technical lemmas}
The results of the following Lemma are well known. However, to the best of our knowledge, we are not able to give a specific reference for the  formulas that we need, even for smooth curves.
The proof will be an easy application of the rules of differential calculus.
\begin{lem}[Geodesic curvature via conformal maps]\label{LemmaComparison}

    Let $f:(M,g)\rightarrow (N,g')$ be a conformal map between two-dimensional oriented Riemannian manifolds such that $g'=e^{\lambda^2 }g$ and let the curve $\gamma:[0,L]\rightarrow M$ be  $\mathcal{C}^2$ and $\gu\in T_{\gamma}M$ the conormal. \\ Suppose that the image curve $c\coloneqq f\circ\gamma:[0,L]\rightarrow N$ is parametrized by arc length. Then if $k_M$ and $k_{N}$ are the geodesic curvature of $\gamma$ and $c$ respectively, for every $t\in (0,L)$ there holds \begin{equation}\label{GCconformal}
        k_N(t)=e^{-\lambda(\gamma(t))}\Big(k_{M}(t)-e^{2\lambda(\gamma(t))}\partial_\gu\lambda(\gamma(t))\Big).
    \end{equation} 
    Moreover, the arc length element $s(t)$ of the curve $\gamma$ is given by \begin{equation}
        s(t)=\int_0^t e^{-\lambda(\gamma(\tau))}\, d\tau.
    \end{equation}
\end{lem}
\begin{proof}
    By definition, the geodesic curvature of the curve $c$ in $N$ is defined as $$k_N=g'(\nabla'_{\dot{c}}\dot{c},\gu'),$$ where  $\gu'$ is the conormal in $N$ such that $\{\dot{c}(t),\gu'(t)\}$ is  an oriented basis of $T_{c(t)}N$ and $\nabla'$ is the Levi-Civita connection on $N$.
    Now, via conformal changes, we have $$e^\lambda\norm{\dot{\gamma}}_{g}=\norm{\dot{c}}_{g'}=1,\quad \quad\norm{\dot{\gamma}}_{g}=e^{-\lambda}$$ 
    and if $\gu(t)$ is the conormal vector in $M$, then $(df)_{\gamma}\gu$ is orthogonal to $\dot{c}$ and $$\norm{(df)_{\gamma}\gu}_{g'}=e^\lambda .$$ We use $\gu'= e^{-\lambda}(df)_{\gamma}\gu $ 
    and by \cite{doCarmo1976} there holds \begin{equation}\label{LC1}
        \nabla'_{(df)_\gamma\dot{\gamma}}(df)_\gamma\dot{\gamma}=(df)_\gamma\left(\nabla_{\dot{\gamma}}\dot{\gamma}\right)+2\dot{\gamma}(\lambda)(df)_\gamma\dot{\gamma}-g(\dot{\gamma},\dot{\gamma})(df)_\gamma(\nabla\lambda),
    \end{equation} and $$\begin{aligned}
       \nabla_{e^{\lambda}\dot{\gamma}}(e^{\lambda}\dot{\gamma})&=e^{\lambda}\left(\nabla_{\dot{\gamma}}(e^{\lambda}\dot{\gamma})\right)=e^{\lambda}\left(e^{\lambda}\nabla_{\dot{\gamma}}\dot{\gamma}+e^{\lambda}\dot{\gamma}(\lambda)\dot{\gamma}\right),
    \end{aligned}$$
    so $$\begin{aligned}
        \nabla_{\dot{\gamma}}\dot{\gamma}=e^{-2\lambda}  \nabla_{e^{\lambda}\dot{\gamma}}(e^{\lambda}\dot{\gamma})-\dot{\gamma}(\lambda)\dot{\gamma}.
    \end{aligned}$$
    Using \eqref{LC1} we get $$\nabla'_{\dot{c}}\dot{c}=(df)_\gamma\left(e^{-2\lambda}  \nabla_{e^{\lambda}\dot{\gamma}}(e^{\lambda}\dot{\gamma})\right)+\dot{\gamma}(\lambda)(df)_\gamma\dot{\gamma}-g(\dot{\gamma},\dot{\gamma})(df)_\gamma(\nabla\lambda) .$$ Finally, we compute the geodesic curvature $k_N$ on $N$ as $$\begin{aligned}
        k_N&=g'(\nabla'_{\dot{c}}\dot{c},\gu')=g'(\nabla'_{\dot{c}}\dot{c},e^{-\lambda}(df)_{\gamma}\gu)\\
&=g'\Big((df)_\gamma\left(e^{-2\lambda}  \nabla_{e^{\lambda}\dot{\gamma}}(e^{\lambda}\dot{\gamma})\right)+\dot{\gamma}(\lambda)(df)_\gamma\dot{\gamma}-g(\dot{\gamma},\dot{\gamma})(df)_\gamma(\nabla\lambda),e^{-\lambda}(df)_{\gamma}\gu\Big)\\
&=e^{-3\lambda}g'\Big((df)_\gamma\nabla_{e^{\lambda}\dot{\gamma}}(e^{\lambda}\dot{\gamma}),(df)_{\gamma}\gu\Big)+e^{-\lambda}\dot{\gamma}(\lambda)g'\Big((df)_\gamma\dot{\gamma},(df)_{\gamma}\gu\Big) \\
& \quad -e^{-\lambda}g'\Big((df)_\gamma(\nabla\lambda),(df)_{\gamma}\gu\Big)\Big)\\
&=e^{-\lambda}g\left(\nabla_{e^{\lambda}\dot{\gamma}}(e^{\lambda}\dot{\gamma}),\gu\right)+e^{\lambda}\dot{\gamma}(\lambda)g\left(\dot{\gamma},\gu\right)-e^{\lambda}g\left(\nabla\lambda,\gu\right)\Big)\\
&=e^{-\lambda}(k_{M}-e^{2\lambda}\partial_\gu\lambda)\,,
    \end{aligned}$$
where we have used that $g(\dot{\gamma},\gu)=0$.
\end{proof}
Our purpose is to use Lemma~\ref{LemmaComparison} locally for a curve  $c:I_\delta\coloneqq[-\delta,\delta]\rightarrow \mathbb{S}_*^2$, parameterized by arc length, pushed forward from a curve $\gamma:I_\delta\rightarrow\R^2$ by a conformal map $f$, where $\quad$ $\mathbb{S}^2_*\coloneqq\mathbb{S}^2\setminus{\{(0,0,-1)\}}$ . More precisely, suppose that $c(0)=(0,0,1)$ and define $\gamma\coloneqq f^{-1}\circ c $, where the map $f$ is given by $$\begin{aligned}&\qquad f:\R^2\longrightarrow \mathbb{S}^2_*
\\
&(x,y)\mapsto \left(\frac{4 x}{4 + x^2 + y^2}, \frac{4 y}{4 + x^2 + y^2}, \frac{8}{4 + x^2 + y^2}-1\right).
\end{aligned} $$
The map $f$ is conformal with conformal factor $$\exp{\lambda(x,y)}=\dfrac{4}{4+x^2+y^2}.$$
We need to apply Lemma~\ref{LemmaComparison} to curves $c$ in the Sobolev space $W^{2,p}(I_\delta,\mathbb{S}^2)$ and $f:\R^2\rightarrow\mathbb{S}^2 _*$. In order to do this, if $\phi$ is a smooth curve, there holds  
$$ \begin{aligned}
&
\qquad\int_{\phi}k_{\mathbb{S}^2}^p(\phi(t))\, dt= \\  &\int_{-\delta}^\delta\left(\dfrac{4+x(t)^2+y(t)^2}{4}\Big(\norm{k_{\R^2}(\gamma)}+\dfrac{C}{(4+x(t)^2+y(t)^2)^2}\gamma(t)\cdot \gu(t)\Big)\right)^p dt,
\end{aligned}  $$ 
where $\gamma=f^{-1}\circ\phi =(x(t),y(t))$, $\gu $ is the conormal, $C$ is a positive constant independent from the curve and $(\cdot)^\top$ stands for the projection of the Euclidean derivative onto the tangent space $T_\gamma\mathbb{S}^2$. By the Dominated Convergence Theorem,  the same formula holds for weak derivatives in $ W^{2,p}$.

\begin{prop}[Arc-length parametrization]\label{PropPLA}
    Fix $c:I_\delta=[-\delta,\delta]\rightarrow \mathbb{S}^2_*$ a curve parametrized by arc-length such that $c(0)=(0,0,1)$.
    Then, the arc-length parametrization $\Gamma(s)=(\tilde{x}(s),\tilde{y}(s))$ of the curve $\gamma(t)\coloneqq f^{-1}\circ c(t)=(x(t),y(t))$ satisfies \begin{equation}\label{cu}
        \int_{I_\delta} k_{\mathbb{S}^2}(c(t))\, dt=\int_0^{\mathcal{L}(\gamma)}\Big(k_{\R^2}(\Gamma(s))+\dfrac{C}{(4+\tilde{x}(s)^2+\tilde{y}(s)^2)^3}\Gamma(s)\cdot \tilde{\gu}(s)\Big)\, ds,
\end{equation}
    where $C$ is a constant independent from the curve and $\gu(t)$ is the conormal vector in $T_{\gamma(t)}\mathbb{S}^2$ orthogonal to $\dot \gamma(t)$, $$s=s(t)\Longleftrightarrow t=\phi(s),$$ and $$\Gamma(s)\coloneqq\gamma(\phi(s))\qquad\Tilde{\gu}(s)\coloneqq \gu(\phi(s)).$$
    Moreover, if $c\in W^{2,p}(I_\delta,\mathbb{S}^2)$, then there exists a non-negative constant $C_\delta\coloneqq C(\delta,f)$ and a positive constant $\tilde{C}_\delta$ close to $1$, independent from the curve, such that \begin{equation}\label{cuu}
         \int_{I_\delta} |k_{\mathbb{S}^2}(c(t))|^p\, dt\geq \int_0^{\mathcal{L}(\gamma)}\tilde{C}_\delta\Big||k_{\R^2}(\Gamma(s))|-C_\delta\delta\Big|^p\, ds \,.
    \end{equation}
\end{prop}
\begin{proof}
Observe that for $t\in I_\delta=[-\delta,\delta]$, the arc element is
    
    $$s(t)\coloneqq \int_{-\delta}^te^{-\lambda(\gamma(t))}\,dt\leq \int_{-\delta}^\delta\dfrac{4+x(t)^2+y(t)^2}{4}=\mathcal{L}(\gamma)$$ 

and if $\delta\leq\sqrt{\epsilon}$ for some $  \epsilon>0$,  then  $\delta\leq \mathcal{L}(\gamma)\leq \delta(1+\epsilon)$.
    By definition, we have $\norm{\dot{\Gamma}}=1$ and $\phi'(s)=e^{\lambda(\Gamma(s))}$. So, 
    $$\hspace{-0.8cm}\begin{aligned}
        \int_{I_\delta} e^{-\lambda(\gamma(t))}\Big(k_{\mathbb{R}^2}(\gamma(t))-e^{2\lambda(\gamma(t))}\partial_\gu\lambda(\gamma(t))\Big)\, dt= \\
 \qquad\qquad       \int_{0}^{\mathcal{L}(\gamma)} e^{-\lambda(\Gamma(s))}\Big(k_{\mathbb{R}^2}(\Gamma(s))-e^{2\lambda(\Gamma(s))}\partial_\gu\lambda(\Gamma(s))\Big)\phi'(s)\, ds=\\
\qquad\qquad        \int_{0}^{\mathcal{L}(\gamma)} \Big(k_{\mathbb{R}^2}(\Gamma(s))-e^{2\lambda(\Gamma(s))}\partial_{\tilde{\gu}}\lambda(\Gamma(s))\Big)\, ds.
    \end{aligned} 
    $$ Moreover, $$\partial_{\tilde{\gu}}\lambda(\Gamma(s))=-\dfrac{1}{2}e^{\lambda(\Gamma(s))}\Gamma(s)\cdot\tilde{n}(s) $$ and \eqref{cu} holds.

We denote by $\tilde{C}_\delta$ a positive constant close to $1$ depending on $\delta$ that can change line by line.
Observe that 
$$\dfrac{1}{\tilde{C}_\delta}\leq e^{\lambda(\Gamma(s))}\leq 1,\quad  \dfrac{\delta}{2}\leq\partial_{\tilde{\gu}}\lambda(\Gamma(s))\leq\dfrac{\tilde C_\delta\delta}{2},$$ and from \eqref{GCconformal}, passing through the arc length parametrization of the curve $\gamma(t)$, denoting with $C$ a positive constant depending only on the conformal map $f$, we have $$
\int_{I_\delta} |k_{\mathbb{S}^2}(c(t))|^p\, dt=\int_0^{\mathcal{L}(\gamma)}e^{(p-1)\lambda(\Gamma(s))}\left|\left(k_{\R^2}(\Gamma(s))+Ce^{3\lambda(\Gamma(s))}\Gamma(s)\cdot \tilde{\gu}(s)\right)\right|^p\, ds\geq $$ $$
\begin{aligned}
&\geq\dfrac{1}{\tilde{C}_\delta}\int_0^{\mathcal{L}(\gamma)}\left|k_{\R^2}(\Gamma(s))+Ce^{3\lambda(\Gamma(s))}\Gamma(s)\cdot \tilde{\gu}(s)\right|^p\, ds\\ &\geq\dfrac{1}{\tilde{C}_\delta}\int_0^{\mathcal{L}(\gamma)}\left||k_{\R^2}(\Gamma(s))|-C\tilde{C}_\delta|\Gamma(s)\cdot \tilde{\gu}(s)|\right|^p\, ds \\&\geq
\dfrac{1}{\tilde{C}_\delta}\int_0^{\mathcal{L}(\gamma)}\left||k_{\R^2}(\Gamma(s))|-2C\tilde{C}_\delta\delta\right|^p\, ds.
\end{aligned}$$ Now, using $C_\delta=2C\tilde{C}_\delta$, we get \eqref{cuu}.
\end{proof}
\begin{lem}
    For $a,b>0$, there exists a constant $C\coloneqq C(p)$ such that
    \begin{equation}\label{pstima}
        |a-b|^p\geq a^p-Cba^{p-1}
    \end{equation}
\end{lem}
\begin{proof}
    Dividing by $b^p$ both terms in \eqref{pstima}, we need to prove that exists a constant $C$ such that $$|t-1|^p\geq t^p-Ct^{p-1} $$ 
for every $t>0$.

Now, the maps $t\mapsto |t-1|^p\ -t^p$ and $t\mapsto t^{p-1}$ are continuous, on every compact set $[0,m]$ with $m\in\R^+$, so there exists a constant $C_m $ for which the inequality holds in $[0,m]$. To have that the constant $C_m$ does not blow up as $m\rightarrow+\infty$, we observe that the limit $$\lim_{t\rightarrow+\infty}\dfrac{(t-1)^p-t^p}{t^{p-1}}=-p.$$  Then, there exists a global constant $C=\max_{m\in[1,\infty]}C_m$ such that \eqref{pstima} holds.
\end{proof}
\newpage
\section{Results}
The aim of this section is to prove the following theorem.
\begin{thm}\label{MainThm}
     Let $c:[0,L]\rightarrow\mathbb{S}^2$ be a rectifiable and open curve in $\mathbb{S}^2$ parametrized in arc-length. Then $$\mathcal{F}_p(c)<\infty \mbox{ for some }p>1\Longleftrightarrow c\in W^{2,p}([0,L],\mathbb{S}^2)$$ and in this case, there holds $$\mathcal{F}_p(c)=\int_c|k_{\mathbb{S}^2}(c(s))|^p\,ds=\int_0^L\norm{\Ddot{c}^\top(s)}^p\,ds.$$
\end{thm}
To prove this Theorem, we split it into two theorems in which we prove an upper and a lower bound on the integral of the geodesic curvature of the curve raised to the power $p$. 
\begin{oss}\label{NoCorners}
   We first observe that the finiteness of the functional $\mathcal{F}_p(c)$ implies that the curve has no corners. Indeed, suppose that the curve has a corner in $c\left(\Bar{t}\right)$ with a turning angle $\theta$, then we construct a polygonal sequence using as partition $\left\{\Bar{t}-1/h,\Bar{t},\Bar{t}+1/h\right\}$ for which $$\mathbf{k}_p(P_h)\geq h^{p-1}\left(\theta_{\Bar{t}}^{(h)}\right)^p$$ where $\theta^{(h)}_{\Bar{t}}$ is the turning angle in $P_h(\Bar{t})$. As $h\rightarrow+\infty$, we have $\theta^{(h)}_{\Bar{t}}\rightarrow\theta$ and $$h^{p-1}\left(\theta_{\bar t}^{(h)}\right)^p\rightarrow +\infty.$$
Then, the curve cannot have any corners if its $p$-curvature is finite.
\end{oss}
\subsubsection*{Upper bound} We start proving the upper bound on the integral of the geodesic curvature raised to the power $p$ of a spherical curve with finite $p$-curvature.
\begin{thm}\label{thm1}
    Let $c$ be a rectifiable and open curve in $\mathbb{S}^2$ parametrized by arc-length such that $\mathcal{F}_p(c)<\infty$ for some $p>1$. Then $c \in W^{2,p}([0,L],\mathbb{S}^2)$ and $$\int_0^L\norm{\Ddot{c}^\top(s)}^p\, ds\leq \mathcal{F}_p(c)<\infty. $$ 

\end{thm}
\begin{proof}
    By finiteness of $\mathcal{F}_p(c)$, we can take a sequence $\{P_h\}$ of polygonal curves inscribed in $c$ satisfying $\mu_c(P_h) \rightarrow 0$ and $\mathbf{k}_p(P_h) \rightarrow \mathcal{F}_p(c)$. For each $h$, let $c_{P_h} : [0, L_h] \rightarrow \mathbb{S}^2$ be the arc-length parametrization of the curve 
    $\gamma(P_h)$ obtained by Construction~\ref{constr}, where $L_h \coloneqq \mathcal{L}(\gamma(P_h))$, and let $ \gamma_h : [0, L]\rightarrow \mathbb{S}^2$ be given by $$\gamma_h(s)\coloneqq c_{P_h}(sL_h/L),$$
    where $L \coloneqq \mathcal{L}(c)$. By piecewise smoothness, apart from a finite set of points, one has 
    $$k_{\mathbb{S}^2}(c_{P_h})(\lambda) = \norm{\Ddot{c_{P_h}}^\top(\lambda)}$$ for $\lambda \in [0, L_h]$ and  $$\Ddot{\gamma}^\top_h(s) = 
    (L_h/L)^2\Ddot{c_{P_h}}(\lambda)+c_{P_h}(\lambda)$$
    for $s \in [0, L]$, with $\lambda = s(L_h/L)$. Therefore, 
\begin{equation}\label{Kp}
        \mathbf{k}_p(P_h)\coloneqq \int_{\gamma(P_h)} |k_{\mathbb{S}^2}(\gamma(P_h))|^p\,ds=\int_0^L\norm{ (L_h/L)^2 \Ddot{c_{P_h}}(\lambda)+c_{P_h}(\lambda)}^p\,d\lambda.
    \end{equation}
   
    Now, we have $d(\gamma(P_h),P_h)\leq\mu_c(P_h)$ for every $h$, whereas $d(P_h,c)\rightarrow 0$; see \cite[Definition 3.1]{MS1}. Since $\mu_c(P_h) \rightarrow 0$, we obtain $d(\gamma (P_h), c) \rightarrow 0$, and hence by lower semicontinuity we infer that $$\mathcal{L}(c)\leq\liminf_{h\rightarrow +\infty}{\mathcal{L}(\gamma(P_h))} .$$
    Using the fact that $\mathcal{L}(\gamma (P_h )) \leq \mathcal{L}(P_h ) \leq \mathcal{L}(c)$ for every $h$, we deduce that $L_h\rightarrow L$. As a consequence, recalling that $\mathbf{k}_p(P_h) \rightarrow \mathcal{F}_p(c)$, by \eqref{Kp} we obtain
    \begin{equation}
       \lim_{h\rightarrow \infty }{\int_0^L\norm{ (L_h/L)^2\Ddot{c_{P_h}}(\lambda)+c_{P_h}(\lambda)}^p\,d\lambda}  \, =\mathcal{F}_p(c).
    \end{equation}    
Since $p > 1$, the sequence $\{\Dot{\gamma}_h\}$ converges strongly in $W^{1,1}$ to some function $v \in W^{1,1}(I_L, \mathbb{S}^2)$. By using that $\gamma_h$ converges to the Lipschitz function $c$ strongly in $L^1(I_L, \mathbb{S}^2)$, we obtain $v = \Dot{c}$ a.e.; hence, possibly passing to a (not relabelled) subsequence, $\{\Dot{\gamma}_h\}$ converges to $\Dot{c}$ weakly in $W^{1,p}(I_L, \mathbb{S}^2)$. In particular, $\Dot{c} \in W^{1,p}(I_L,\mathbb{S}^2 ) $ and
$\Ddot{\gamma}_h$ converges to $\Ddot{c}$ weakly in $L^p$  and, using that the normal component is the curve $c$ itself, we have that $\Ddot{\gamma}^\top_h$ converges to $\Ddot{c}^\top$ weakly in $L^p$.\\ Hence the curve $c$ is $\mathcal{C}^1$ and, by lower semicontinuity,
$$\int_0^L\norm{\Ddot{c}^\top(s)}^p\, ds\leq \mathcal{F}_p(c)<\infty. $$ 
\end{proof}
We observe that, supposing $c\in W^{2,p}(I,\mathbb{S}^2)$, we have the following absolutely continuity property.

\begin{oss}\label{AC}
    Suppose that the spherical curve parametrized by arc length $c:I_L\rightarrow\mathbb{S}^2$ is in the Sobolev space $W^{2,p}$, then the functional $$\int_J |k_{\mathbb{S}^2}(c(t))|^p\,dt$$ is absolutely continuous. Then, for every $\epsilon>0$, there exists $\delta_{AC}\coloneqq\delta_{AC}(\epsilon)$ such that if an interval $J\subseteq I$ is small, i.e., $|J|<\delta_{AC}$, then $$\int_J|k_{\mathbb{S}^2}(c(t))|^p\,dt\leq \epsilon. $$ Moreover, for every $q\in   [1,p]$, we have $$\left(\dashint_J|k_{\mathbb{S}^2}(c(t))|^{q}\,dt\right)^{\tfrac{p}{q}}\leq\dashint_J|k_{\mathbb{S}^2}(c(t))|^{p}\,dt,$$ and \begin{equation}\label{ACstima}
        \int_J|k_{\mathbb{S}^2}(c(t))|^{q}\,dt\leq\delta_{AC}^{1-q/p}\left(\int_J|k_{\mathbb{S}^2}(c(t))|^{p}\,dt\right)^{\tfrac{q}{p}}\leq \epsilon^{q/p}\delta_{AC}^{\tfrac{p-q}{p}}. 
    \end{equation}

\end{oss}
\subsubsection*{Lower bound}
Now, we are able to prove the following lower bound. The idea is to localize the Sobolev spherical curve, use the  comparison result given by Proposition~\ref{PropPLA}, and \cite[Theorem 5.3]{MS1} for the Euclidean comparison curve.

\begin{thm}\label{lowerbound}
Let $c$ be a rectifiable and open curve in $\mathbb{S}^2$ parametrized by arc-length of class  $W^{2,p}([0,L],\mathbb{S}^2)$ for some $p>1$. 
Then,  for every $\epsilon>0$ small, there exists a sequence $\{P_h\}$ of polygonal curves inscribed in $c$ and a constant $c_\epsilon$ such that $\mu_c(P_h) \rightarrow 0 $, $c_\epsilon\rightarrow 1$ as $\epsilon\rightarrow 0$ and
\begin{equation}
    c_\epsilon\mathbf{k}_p(P_h)-g(\epsilon)\leq \int_0^L\vert k_{\mathbb{S}^2}(c(s)) \vert^p\,ds<\infty,
\end{equation}
where $g(\epsilon)\rightarrow 0$ as $\epsilon\rightarrow 0$.
 Therefore $$\mathcal{F}_p(c)\leq\int_0^L\vert k_{\mathbb{S}^2}(c(s))\vert^p\, ds.$$ 

\end{thm}

\begin{proof}
  Fix $\epsilon>0$.

\textbf{Step 1.} We find upper and  lower bounds on the length of the edges of an equilateral polygonal inscribed in the curve $c$ and we analyze the $p$-curvature functional of this polygonal.

We start producing an equilateral polygonal curve inscribed in $c$. 
We observe that there exists $\delta_1\coloneqq \delta_1(\epsilon)$ such that if the edges of any equilateral polygonal, inscribed in $c$, have length $\ell$ less than $\delta_1$, then 
\begin{equation}\label{ConstrTempi}
    \ell(1+\epsilon)\geq |t_{i+1}-t_i|\geq \ell.
\end{equation}
This is implied by the regularity of the curve $c$ and using that it is parametrized by arc length. Indeed, $c$ is $1$-Lipschitz and $$\ell=d_{\mathbb{S}^2}(c(t_i),c(t_{i+1}))\leq |t_{i+1}-t_i|.$$ For the other inequality, we use that $$\dfrac{d_{\mathbb{S}^2}(c(t_{i+1}),c(t_{i}))}{|t_{i+1}-t_{i}|}\rightarrow 1,$$ and by the uniformly continuity of the (metric) derivative of $c$, we have the existence of such $\delta_1$.

Fix $\tilde \ell=\tfrac{\delta}{2}$  where $\delta\coloneqq\min\left\{\delta_1,\delta_{AC},\sqrt{\epsilon}\right\}$ and let $h$ be the unique integer for which there is a sequence of times $t_0=0<\cdots<t_h=L$ such that $d_{\mathbb{S}^2}(c(t_{i-1}),c(t_i))=\tilde\ell$ for $i=1,\ldots, h-1$ and $d_{\mathbb{S}^2}(c(t_{h-1}),c(t_h))=\hat\ell<\tilde \ell$. Then, we have a partition $\mathcal{P}_h$ of $[0,L]$ of the form 
 $$\mathcal{P}_h=\{t_0,\ldots,t_h\},\quad I_i=[t_{i-1},t_{i}],\quad|I_i|=\tilde\ell_i
$$ where the constraint \begin{equation}\label{Constr.Lunghezze}
    \sum_{i=1}^h\tilde{\ell}_h=L
\end{equation}  holds. If $\hat\ell=\tilde\ell$, we end up with an equilateral polygonal on $\mathcal{P}_h$ choosing $\ell=\tilde \ell$. If this is not the case, consider $$\ell\coloneqq\max\Big\{ l\leq\delta/2\Big|\,\exists\, \{0=t_0<\cdots<t_h=L\}: d_{\mathbb{S}^2}(c(t_i),c(t_{i+1}))=l   \Big\},$$ where the maximum exists because we have fixed the number of edges $h$ and the length $\ell$ can be obtained enlarging the length $\hat \ell$ of the last edge and shrinking uniformly the length of the first $h-1$ edges. We observe that the length $\ell$ must be greater of $\delta/3$. Indeed, supposing by contradiction that $\ell<\delta/3$, from \eqref{ConstrTempi} and \eqref{Constr.Lunghezze} we have the inequalities
$$ h\delta/3 (1+\epsilon)\geq\sum_{i=1}^h\tilde{\ell}_h=L\geq(h-1)\delta/2+\tilde\ell_h$$ that makes $\tilde\ell_h$ negative, which cannot hold.

We observe, that from $\ell\geq \delta/3$ and $h\ell(1+\epsilon)\leq L$, we obtain \begin{equation}\label{hconstr}
    h\leq \dfrac{3\tilde{c}_\epsilon L}{\delta(1+\epsilon)}.
\end{equation}

Moreover, the obtained equilateral polygonal $P_h=P(\mathcal{P}_h)$ with turning angles $\theta_i$ in $\quad$ $P_h(t_i)=c(t_i)$ for $i=1,\ldots,h-1$ has small angles $\theta_i$ if $\delta$ is small, indeed localizing around $P(t_i)$ and using Proposition~\ref{PropPLA} one has $\theta_i\leq\delta (1+\epsilon)$ and $$\tan(\theta_i/2)\leq(1+\epsilon)\theta_i/2.$$ 
We define as $\Psi(\ell,\theta)$  the continuous function $$\Psi(\ell,\theta)\coloneqq \sqrt{\sin^2(\ell/2)+\sin^2(\theta/2)\cos^2(\ell/2)}.$$ From $\ell<\sqrt{2\epsilon}$, we have $\sin(\ell)\geq \ell(1-\ell^2/6)\geq \ell(1-\epsilon)$  and  $\cos(\ell)\geq (1-\epsilon)$.\\
From now on, we denote by $\tilde{c}_\epsilon$ a constant that can change line by line such that $\tilde{c}_\epsilon\rightarrow 1$ as $\epsilon\rightarrow 0$.

We first bound, for every $h$, the $p$-rotation $\mathbf{k}_p(P_h)$ of such polygonal from above, namely   
 $$\begin{aligned}
\mathbf{k}_p(P_h)&= \sum_{i=1}^{h-1}2\arctan\left(\frac{\Psi(\theta,\ell)}{\cos(\theta_i/2)\cos(\ell/2)}\right)\cdot \frac 1{\Psi(\theta,\ell)}\cdot \frac{\sin^p(\theta_i/2)}{(\cos(\theta_i/2)\sin(\ell/2))^{p-1}}\,\\    &\leq \sum_{i=1}^{h-1} 2 \frac{\tan^p(\theta_i/2)}{\cos(\ell/2)\sin^{p-1}(\ell/2)}\leq\sum_{i=1}^{h-1} 2^{p}\tilde{c}_\epsilon\tan^p(\theta_i/2)\ell^{1-p} 
\end{aligned}$$
whence we estimate
\begin{equation}\label{primadis}
    \mathbf{k}_p(P_h)\leq \tilde{c}_\epsilon\sum_{i=1}^{h-1} \ell^{1-p}\theta_i^p.
\end{equation}
\textbf{Step 2.} We bound from below the integral of the $p$-power of the geodesic curvature.

Let $\psi:[0,h\ell]\rightarrow [0,L]$ be a reparametrization such that $\psi(s_i)=t_i$ and $\psi$ is affine in $[s_i,s_{i+1}]$, with $$\psi'(s)_{\Big|(s_i,s_{i+1})}=\dfrac{t_{i+1}-t_i}{s_{i+1}-s_i}.$$

    Following \cite{Bruckstein01112001}, we have $$ 
   \int_0^L |k_{\mathbb{S}^2}(c(t))|^p\, dt\geq\sum_{i=1}^{h-1}\dfrac{1}{\ell}\int_0^\ell\left(\int\limits_{\psi(s_{i-1}+a)}^{\psi(s_{i}+a)}|k_{\mathbb{S}^2}(c(t))|^p\, dt\right) \,da.$$

   Using Proposition~\ref{PropPLA}, we compare the geodesic curvature of the localizing curves $$c_i\coloneqq c_{|_{[\psi(s_{i-1}+a),\psi(s_{i}+a)]}},$$ with the curvature in $\R^2$ of the 
comparison curves $\Gamma_i(s)$, namely 
$$\int\limits_{\psi(s_{i-1}+a)}^{\psi(s_{i}+a)}\vert k_{\mathbb{S}^2}(c_i(t)) \vert^p\, dt=\int\limits_0^{\mathcal{L}(a,i)}e^{(p-1)\lambda(\Gamma_i(s))}\left\vert k_{\R^2}(\Gamma_i(s))+\dfrac{1}{2}e^{3\lambda(\Gamma_i(s))}\Gamma_i(s)\cdot \tilde{\gu}_i(s)\right\vert^p\, ds,    
    $$ where $\mathcal{L}(a,i)\coloneqq\mathcal{L}\left(\gamma_{i{|_{[\psi(s_{i-1}+a),\psi(s_{i}+a)]}}}\right)\leq \tilde c_\epsilon\delta$. 
   By \eqref{cuu}, we obtain 
$$ 
        \int\limits_0^{\mathcal{L}(a,i)}e^{(p-1)\lambda(\Gamma_i)}\left\vert k_{\R^2}(\Gamma_i)+\dfrac{1}{2}e^{3\lambda(\Gamma_i)}\Gamma_i(s)\cdot \tilde{\gu}(s)\right\vert^p\, ds\geq\int\limits_0^{\mathcal{L}(a,i)}\tilde{c}_\epsilon\Big||k_{\R^2}(\Gamma_i)|-C_\delta\delta\Big|^p\, ds.
$$
Using \eqref{pstima}, we obtain $$\tilde{c}_\epsilon\int\limits_0^{\mathcal{L}(a,i)}\Big||k_{\R^2}(\Gamma_i)|-C_\epsilon\delta\Big|^p\, ds\geq\tilde{c}_\epsilon\int\limits_0^{\mathcal{L}(a,i)}|k_{\R^2}(\Gamma_i)|^p\,ds-C\delta\int\limits_0^{\mathcal{L}(a,i)}|k_{\R^2}(\Gamma_i)|^{p-1}\,ds, $$ where $C\coloneqq C(\epsilon,p)$ is a positive constant that is bounded as $\epsilon\rightarrow 0$, that can change line by line.
    We denote by $I_1,I_2$
    $$ \begin{aligned}
        &I_1\coloneqq\sum_{i=1}^{h-1}\tilde{c}_\epsilon\dashint_0^\ell\left(\int\limits_0^{\mathcal{L}(a,i)}|k_{\R^2}(\Gamma_i(s))|^p\right)\,da\\
        &I_2\coloneqq\sum_{i=1}^{h-1}C\delta\dashint_0^\ell \left(\int\limits_0^{\mathcal{L}(a,i)}|k_{\R^2}(\Gamma_i(s))|^{p-1}\,ds\right)\,da
    \end{aligned}$$
    We first deal with $I_1$.
    
    In the Euclidean setting, if the curve $\Gamma_i$ is parametrized by arc length, then the curvature is given by the norm of the second derivative $\Gamma_i''(s)$ of the curve, i.e., 
   $$
\int\limits_0^{\mathcal{L}(a,i)}|k_{\R^2}(\Gamma_i(s))|^p\, ds=\int\limits_0^{\mathcal{L}(a,i)}\norm{\Gamma_i^{''}(s)}^p\, ds\,. $$  
Applying twice the Jensen inequality and supposing $t_{i+1}-t_i\geq t_i-t_{i-1}$, we have 
$$   \begin{aligned}
       \tilde{c}_\epsilon\dashint_0^\ell   \left(\int\limits_0^{\mathcal{L}(a,i)}\norm{\Gamma_i^{''}(s)}^p\, ds\right)\,da&\geq\tilde{c}_\epsilon\dashint_0^\ell \| \Gamma_i^{'}(\psi(s_{i}+a))-\Gamma_i^{'}(\psi(s_{i-1}+a)) \|^p\, da\\ 
      & \geq  \tilde{c}_\epsilon\ell^{1-p}\left\|\int_0^\ell\left(\Gamma_i^{'}(\psi(s_{i}+a))-\Gamma_i^{'}(\psi(s_{i-1}+a))\right)\, da\right\|^p\\ & \geq\tilde{c}_\epsilon\ell^{1-p}\left\|  \dfrac{\Gamma_i(t_{i+1})-\Gamma_i(t_i)}{t_{i+1}-t_i}- \dfrac{\Gamma_i(t_{i})-\Gamma_i(t_{i-1})}{t_{i}-t_{i-1}}\right\|^p.
   \end{aligned}$$
   Finally, using \cite[Theorem 5.3]{MS1}, we have $$\tilde{c}_\epsilon\dashint_0^\ell\int_0^{\mathcal{L}(a,i)}\norm{\Gamma_i^{''}(s)}^p\, ds\,da \geq  \tilde{c}_\epsilon\ell^{1-p}\theta_i^p.$$  Thanks to the conformality of the projection, the turning angles $\theta_i$ for $i=1,\ldots ,h-1$ are exactly the turning angles of $P_h$ and by \eqref{primadis} there holds
   $$I_1\geq  \tilde{c}_\epsilon\sum_{i=1}^{h-1}\ell^{1-p}\theta_i^p\geq \tilde{c}_\epsilon\mathbf{k}_p(P_h).$$ 
   
   Now, we prove that $I_2\leq g(\epsilon)\rightarrow 0$ as $\epsilon\rightarrow 0$.\\
   By Remark~\ref{AC}, we have 
$$\int\limits_0^{\mathcal{L}(a,i)}\left\vert k_{\R^2}(\Gamma_i(s))\right\vert^{q}\,ds \leq\epsilon^{q/p}\delta^{(p-q)/p}.$$
Then, using $q=p-1$ we obtain
$$\begin{aligned}
    I_2=&\sum_{i=1}^{h-1}C\delta\left(\dashint_0^\ell \int\limits_0^{\mathcal{L}(a,i)}|k_{\R^2}(\Gamma_i(s))|^{p-1}\,ds\right)\,da\\
    \leq&\sum_{i=1}^{h-1}C\delta\epsilon^{(p-1)/p}\delta^{1/p}
\end{aligned} $$
Summing over the index $i$ and using \eqref{hconstr}, we get 
$$\begin{aligned}
    &\sum_{i=1}^{h-1}C\epsilon^{(p-1)/p}\delta^{1+\tfrac{1}{p}}\leq C\epsilon^{(p-1)/p}\delta^{\tfrac{1}{p}}\coloneqq g(\epsilon).
\end{aligned}$$
   Then, putting all together
   $$\int_0^L|k_{\mathbb{S}^2}(c(s))|^p=I_1-I_2\geq\tilde{c}_\epsilon\mathbf{k}_p(P_h)-g(\epsilon).$$
   Now the theorem follows letting $P_h=P_{\epsilon_h}$ for a suitable decreasing sequence $\epsilon_
   h\rightarrow 0$.

\end{proof}

\appendix
\renewcommand\sectionname{Appendix}
\setcounter{section}{0}
\renewcommand{\thesection}{\Alph{section}}
\section{Construction of spherical bends}\label{AA}
In this appendix, we explain how to construct from a given polygonal $P$ with vertexes $P(t_i)$ and turning angles $\theta_i$ for $i=1,\ldots,h-1$ a local bending curve around $P(t_i)$ that glues $\mathcal{C}^1$ with $P$ near to $t_i$ with constant geodesic curvature. Working locally, up to isometry, we can suppose that the polygonal near $P(t_i)$ is given by $$ P(t)=\begin{cases}
    (-K\sin t,S\sin t,\cos t) & t\in[-\delta,0] \\
(K\sin t,S\sin t,\cos t) & t\in[0,\delta]\,
.\end{cases}$$
where $K:=\cos\alpha_i$ and $S:=\sin\alpha_i$, with $\alpha_i=(\pi-\theta_i)/2$.
We want to perform the construction only on half of each geodesic segment with $\delta=\ell/2.$
\begin{constr}\label{constr}
    
    We start with the localized polygonal curve 
$$ P(t)=\begin{cases}
    (-K\sin t,S\sin t,\cos t) & t\in[-\ell/2,0] \\
(K\sin t,S\sin t,\cos t) & t\in[0,\ell/2]\, 
.\end{cases}$$
where $K:=\cos\alpha$ and $S:=\sin\alpha$, with $\alpha=(\pi-\theta)/2$.
Then,
$$ P'(t)=\begin{cases}
    (-K\cos t,S\cos t,-\sin t) & t\in(-\ell/2,0) \\
(K\cos t,S\cos t,-\sin t) & t\in(0,\ell/2)\,.
\end{cases}  $$
Denote $P_{\pm}:=P(\pm\ell/2)$ and $v_\pm:=P'(\pm\ell/2)$, so that with $s:=\sin(\ell/2)$ and $c=\cos(\ell/2)$ we have:
$$ P_-=(Ks,-S s,c)\,,\quad P_{+}=(Ks,S s,c)\,,\quad v_-=(-Kc,S c,s)\,,\quad v_+=(Kc,S s,-s)\,. $$
We choose a rotation matrix $R\in SO(3)$ with rotational axis $e_2$ and angle $\beta$. With $\tau=\cos\beta$ and $\sigma=\sin\beta$, we have
$$ v_-R^T=(-(\tau Kc+\sigma s), S c, -(\sigma K c-\tau s))\,,\quad v_+R^T=((\tau Kc+\sigma s), S c, (\sigma K c-\tau s)) $$
and choose $\beta$ so that $\sigma K c-\tau s=0$, i.e.,
$$ \tau=\frac{Kc}{\sqrt{s^2+K^2c^2}}\,,\quad \sigma=\frac{s}{\sqrt{s^2+K^2c^2}}\,. $$
This way, the rotated velocity vectors $V_\pm R^T$ at the end points have zero third components. We correspondingly have
$$ P_-R^T=(x,-y,z)\,, \quad P_+R^T=(x,y,z)\,,\quad x:=\tau Ks-\sigma c\,,\quad y:=S s\,,\quad z:=\sigma Ks+\tau c\,, $$ 
so that explicitly
$$ x=-\frac{S^2sc}{\sqrt{s^2+K^2c^2}}\,,\quad z=\frac{K}{\sqrt{s^2+K^2c^2}}\,. $$
We now choose the angle $\Phi\in(0,\pi)$ so that
$$ \cos\Phi=z\,,\quad \sin\Phi=\sqrt{x^2+y^2}=\frac{-S s}{\sqrt{s^2+K^2c^2}}\,. $$
This way, we reduce to compute the integral of the $p$-th power of the curvature of the parallel $\gamma(P)$ connecting the rotated points $P_\pm R^T$. We set
$$ \gamma(P)(s)=(\sin\Phi\cos s,\sin\Phi\sin s,\cos\Phi) $$
and imposing that $\gamma(P)(s_0)=(-x,y,z)$, we infer that $\sin s_0=\sqrt{s^2+K^2c^2}$ and $\cos s_0=-S c$. Therefore, the curve $\gamma(P)$ has length
$$ \mathcal{L}(\gamma(P))=2\arctan\left(\frac{\sqrt{s^2+K^2c^2}}{-S c}\right)\cdot R\,,\quad R=\sin\Phi\,. $$
Moreover, the geodesic curvature density of $\gamma(P)$ is equal to the constant
$$ k_{\mathbb{S}^2}=\cot\Phi=\frac{K}{-S s}\,, $$
and we definitely obtain
$$ \int_{\gamma(P)}|k_{\mathbb{S}^2}|^p\,ds=2\arctan\left(\frac{\sqrt{s^2+K^2c^2}}{-S c}\right)\cdot \frac{S s}{\sqrt{s^2+K^2c^2}}\cdot \left(\frac{K}{-S s}\right)^p\,. $$
Replacing $K=\sin(\theta/2)$ and $S=-\cos(\theta/2)$, we finally get for every $p\geq 1$ the quantity
$$ \mathcal{F}_p(\ell,\theta)\coloneqq 2\arctan\left(\frac{\Psi(\ell,\theta)}{\cos(\theta/2)\cos\ell/2}\right)\cdot \frac 1{\Psi(\ell,\theta)}\cdot \frac{\sin^p(\theta/2)}{(\cos(\theta/2)\sin\ell/2)^{p-1}},$$ where $$h(\ell,\theta)\coloneqq\sqrt{\sin^2\ell/2+\sin^2(\theta/2)\cos^2\ell/2}.$$
    We conclude observing that this quantity is preserved under isometries.
\end{constr}
\nocite{MS2b}


\begin{thebibliography}{ZZZ}
\normalsize{
\bibitem[AM17]{AceMuc16}
Emilio Acerbi and Domenico Mucci.
\newblock Curvature-dependent energies: the elastic case.
\newblock {\em Nonlinear Analysis}, 153:7--34, 2017.

\bibitem[AleR12]{alexandrov2012general}
V.~V. Alexandrov and Y.~G. Reshetnyak.
\newblock {\em General Theory of Irregular Curves}.
\newblock Mathematics and its Applications. Springer Netherlands, 2012.

\bibitem[Bru01]{Bruckstein01112001}
A.~M. Bruckstein, A.~N. Netravali, and T.~J. Richardson.
\newblock Epi-convergence of discrete elastica.
\newblock {\em Applicable Analysis}, 79(1-2):137--171, 2001.

\bibitem[DC76]{doCarmo1976}
Manfredo~P. do~Carmo.
\newblock {\em Differential Geometry of Curves and Surfaces}.
\newblock 1976.

\bibitem[Cas]{C}
M.~Castrill{\'o}n L{\'o}pez, V.~Fern{\'a}ndez Mateos, and J.~Mu{\~n}oz Masqu{\'e}.
\newblock Total curvature of curves in {R}iemannian manifolds.
\newblock {\em Differential Geom. Appl.}, 28(2):140--147, 2010.


\bibitem[Dek79]{D}
B.~V. Dekster.
\newblock Upper estimates of the length of a curve in a {R}iemannian manifold
  with boundary.
\newblock {\em J. Differential Geom.}, 14(2):149--166, 1979.

\bibitem[Mil50]{Mil}
J.~W. Milnor.
\newblock On the total curvature of knots.
\newblock {\em Ann. of Math.}, 52(2):248--257, 1950.

\bibitem[Mil53]{Mi53}
J.~W. Milnor.
\newblock On total curvatures of closed space curves.
\newblock {\em Mathematica Scandinavica}, 1:289--296, 1953.

\bibitem[ManL03]{ML}
Chaiwat Maneesawarng and Yongwimon Lenbury.
\newblock Total curvature and length estimate for curves in {CAT}({K}) spaces.
\newblock {\em Differential Geom. Appl.}, 19(2):211--222, 2003.

\bibitem[MS23]{MS1}
Domenico Mucci and Alberto Saracco.
\newblock Weak elastic energy of irregular curves.
\newblock {\em Philos. Trans. Roy. Soc. A}, 381, 2023.

\bibitem[MS21a]{MS2}
Domenico Mucci and Alberto Saracco.
\newblock The total intrinsic curvature of curves in {R}iemannian surfaces.
\newblock {\em Rend. Circ. Mat. Palermo (2)}, 70:521--557, 2021.

\bibitem[MS21b]{MS2b}
Domenico Mucci and Alberto Saracco.
\newblock Correction to: The total intrinsic curvature of curves in {R}iemannian
  surfaces.
\newblock {\em Rend. Circ. Mat. Palermo (2)}, 70:1137--1138, 2021.
\bibitem[Res05]{Re}
Yu.~Reshetnyak.
\newblock The theory of curves in differential geometry from the viewpoint of
  the theory of functions of a real variable.
\newblock {\em Russian Math. Surveys}, 60:1165--1181, 2005.

\bibitem[Sul06]{Su}
John Sullivan.
\newblock Curves of finite total curvature.
\newblock {\em Discrete Differential Geometry}, 38:137--161, 2006.

}
\end{thebibliography}
\end{document}